\documentclass[a4paper,12pt]{article}
\usepackage[utf8]{inputenc}
\usepackage{amsthm}
\usepackage{amsfonts}
\usepackage{amssymb}
\usepackage{hyperref}
\usepackage{graphicx}
\usepackage{rotating}
\usepackage{wrapfig}
\usepackage[all]{xy}
\usepackage{pgf,tikz}
\usetikzlibrary{arrows}
\usetikzlibrary[patterns]

\newcommand{\N}{\mathbb{N}}

\newcommand{\R}{\mathbb{R}}

\newtheorem{prop}{Proposition}[section]
\newtheorem{cor}[prop]{Corollary}

\newtheorem{teo}[prop]{Theorem}

\newtheorem{lema}[prop]{Lemma}
\newtheorem{conj}[prop]{Conjecture}

\theoremstyle{definition}

\newtheorem*{ex}{Example}
\newtheorem{defi}[prop]{Definition}
\newtheorem*{obs}{Remark}
\newtheorem*{acknowledgements}{Acknowledgements}

\title{Constructing Bowditch boundaries of some relatively hyperbolic groups that are homeomorphic to the $n$-dimensional Sierpi\'nski carpet}
\author{Lucas H. R. de Souza}

\begin{document}

\DeclareGraphicsExtensions{.pdf,.jpg,.mps,.png,}

\maketitle

\def\eod{\hfill$\square$}

\begin{abstract}In this paper we prove that if some relatively hyperbolic groups have Bowditch boundary homeomorphic to the $n$-sphere, then they are also relatively hyperbolic with respect to another set of parabolic subgroups and its Bowditch boundary is homeomorphic to the $n-1$-dimensional Sierpi\'nski carpet.
\end{abstract}

\let\thefootnote\relax
\let\thefootnote\relax

\tableofcontents

\section*{Introduction}

Tshishiku and Walsh \cite{TW} (with a correction in \cite{TW2}) showed the following topological characterization of the $1$-dimensional Sierpi\'nski Carpet: if we take a countable dense subset of the $2$-sphere and, for each point $p$ on this set, we take a quotient of the closed $2$-disk to the $2$-sphere that collapses the boundary to the point $p$, then the inverse limit of this family of maps is homeomorphic to the $1$-dimensional Sierpi\'nski Carpet. Later, we generalized their result \cite{So4} for $n-1$-dimensional Sierpi\'nski Carpets, with $n \neq 4$. More classical characterizations of the Sierpi\'nski Carpet are given in \cite{Wh} and \cite{Ca} (where the same problem with the dimension appears).

In this paper we use such characterization of the Sierpi\'nski Carpet in order to show that for some classes of relatively hyperbolic pairs of groups with their Bowditch boundaries homeomorphic to the $n$-sphere, we can refine their peripheral structures in a way that their new Bowditch boundaries are homeomorphic to the $n-1$-dimensional Sierpi\'nski Carpet. Precisely, we prove the following:

\

$\!\!\!\!\!\!\!\!\!\!\!\!$ \textbf{Theorem \ref{sierpinskihyp}} Let $(G,\mathcal{P})$ be a relatively hyperbolic pair with Bowditch boundary homeomorphic to $S^{n}$ and $\mathcal{Q}$ a proper subset of $\mathcal{P}$ such that every group in $\mathcal{P} - \mathcal{Q}$ is hyperbolic. Suppose that one of the following holds:

\begin{enumerate}
    \item $n = 1$ or $2$.
    \item $n = 3$ and every group in $\mathcal{P}-\mathcal{Q}$ is virtually torsion-free.
    \item $n = 4$,  the topological characterization of the $3$-dimensional Sierpi\'nski Carpet as an inverse limit of $4$-balls holds, every group $Q \in \mathcal{P} - \mathcal{Q}$ is torsion-free, it acts properly discontinuously, cocompactly and by isometries on a simply connected manifold $M_{Q}$ with negative curvature and the Borel Conjecture holds for $Q$.
    \item $n \geqslant 5$ and every group in $\mathcal{P} - \mathcal{Q}$ has its hyperbolic boundary homeomorphic to the $n-1$-sphere.
\end{enumerate}

Then the Bowditch boundary of $G$ with respect to $\mathcal{Q}$ is homeomorphic to the $n-1$-dimensional Sierpi\'nski carpet.

\

The methods used here are similar to the one used by Tshishiku and Walsh to show that the Dahmani boundary of a relatively hyperbolic $(G,\mathcal{P})$ is homeomorphic to a $1$-dimensional Sierpi\'nski carpet, given that its Bowditch boundary is homeomorphic to the $2$-sphere (Theorem 1.1 of \cite{TW}). In  fact, if the Dahmani boundary is a blowup space (see \textbf{Section 3.1.2} for the definition), then their theorem has ours as a consequence, in the case of dimension $2$.

Let $p \in S^{n}$ be a bounded parabolic point. Its stabilizer $H$ acts on $S^{n}-\{p\}$ properly discontinuously and cocompactly and then we are able to have a glueing between $S^{n}-\{p\}$ and the hyperbolic boundary of $H$. To prove the theorem, it is necessary to show that this glueing is homeomorphic to the closed $n$-ball (and then we can use the characterization of the Sierpi\'nski Carpet described above). For different dimensions, we use different tools to show it.

It is not much clear what really happens in dimension $4$, as it is also not clear if the characterization of the $3$-dimensional Sierpi\'nski Carpet holds. Furthermore, as far as I know, the Borel Conjecture is not well understood in dimension $4$, unlike any other dimension.

As consequences of the theorem above, we have the following:

\begin{enumerate}
    \item Another proof of the well known fact that if $(G,\mathcal{P})$ is a relatively hyperbolic pair such that $\mathcal{P}$ is not empty and its Bowditch boundary is homeomorphic to $S^{1}$, then $G$ is virtually free (\textbf{Corollary \ref{circleimpliesvirtuallyfree}}).
    \item Let $(G,\mathcal{P})$ be a relatively hyperbolic pair with Bowditch boundary homeomorphic to $S^{2}$ and such that $\mathcal{P}$ is not empty and every group in $\mathcal{P}$ is hyperbolic. If the Kapovich-Kleiner Conjecture holds, then $G$ is virtually the fundamental group of a compact hyperbolic $3$-manifold with non-empty totally geodesic boundary (\textbf{Corollary \ref{conjecturedependent}}).
\end{enumerate}

The Kapovich-Kleiner Conjecture says that if a hyperbolic group G has its hyperbolic boundary homeomorphic to the $1$-dimensional Sierpinski carpet, then $G$ is virtually the fundamental group of a compact hyperbolic $3$-manifold with non-empty totally geodesic boundary \cite{KK}.

\begin{acknowledgements}This paper contains part of my PhD thesis. It was written under the advisorship of Victor Gerasimov, to whom I am grateful. I also would like to thank Christopher Hruska, for his corrections on the proof of the main theorem.
\end{acknowledgements}

\section{Preliminaries}

\subsection{The perspectivity property}

\subsubsection{Artin-Wraith glueings}

This theory is covered in \cite{So}. We need to use it in the following sections just as a language to define some compactifications of spaces.

\begin{defi}Let $X$ and $Y$ be topological spaces and an application $f: Closed(X) \rightarrow Closed (Y)$ such that $\forall A,B \in Closed(X), \ f(A\cup B) = f(A)\cup f(B)$ and $f(\emptyset) = \emptyset$ (we will say that such map is admissible). We will give a topology for $X\dot{\cup} Y$. Let's declare as a closed set $A \subseteq X\dot{\cup} Y$ if $A \cap X \in Closed(X), \ A \cap Y \in Closed(Y)$ and $f(A\cap X) \subseteq A$. Therefore, let's denote by $\tau_{f}$ the set of the complements of this closed sets and $X+_{f}Y = (X\dot{\cup} Y,\tau_{f})$.
\end{defi}

Every space $X$ that is a union of disjoint subspaces $A$ and $B$, where $A$ is open on $X$, can be recovered uniquely as an Artin-Wraith glueing of $A$ and $B$.

\begin{defi}Let $X+_{f} \! Y$ and $Z+_{h} \! W$ be topological spaces and continuous maps $\psi: X \rightarrow Z$ and $\phi: Y \rightarrow W$. We define $\psi + \phi: X+_{f}Y \rightarrow Z+_{h}W$ by $(\psi + \phi)(x) = \psi(x)$ if $x \in X$ and $\phi(x)$ if $x \in Y$. If $G$ is a group, $\psi: G \curvearrowright X$ and $\phi: G \curvearrowright Y$, then we define $\psi+\phi: G \curvearrowright X+_{f}Y$ by $(\psi+\phi)(g,x) = \psi(g,x)$ if $x \in X$ and $\phi(g,x)$ if $x \in Y$.
\end{defi}

\begin{defi}Let $X+_{f}W$, $Y$ and $Z$ be topological spaces, $\pi: Y \rightarrow X$ and $\varpi: Z \rightarrow W$ be two continuous maps. We define the pullback of $f$ with respect to $\pi$ and $\varpi$ by $f^{\ast}(A) = \varpi^{-1}(f(Cl_{X}\pi(A)))$.
\end{defi}

\begin{cor}\label{continuousidentity}Let $X+_{f}Y$, $X+_{g}Y$ be topological spaces. Then, the map $id: X+_{f}Y \rightarrow X+_{g}Y$ is continuous if and only if $\forall A \in Closed (X)$, $f(A) \subseteq g(A)$.
\eod\end{cor}

\subsubsection{Perspective compactifications}

Most of the following are covered in \cite{So2}.

\begin{defi}\label{pielambda}Let $G$ be a group, $X$ and $Y$ Hausdorff topological spaces with $X$ locally compact and $Y$ compact, $L: G \curvearrowright G$ the left multiplication action, $\varphi: G \curvearrowright X$ a  properly discontinuous cocompact action, $\psi: G \curvearrowright Y$ an action by homeomorphisms and $K \subseteq X$ a compact such that $\varphi(G,K) = X$. Define $\Pi_{K}: Closed(X) \rightarrow  Closed(G)$ as $\Pi_{K}(S) = \{g\in G: \varphi(g,K)\cap S \neq \emptyset\}$ and $\Lambda_{K}: Closed(G) \rightarrow Closed(X)$ as $\Lambda_{K}(F) = \varphi(F,K)$.
\end{defi}

\begin{defi}Let $G$ be a group, $X$ a Hausdorff locally compact space and $\varphi: G \curvearrowright X$ a properly discontinuous and cocompact action. We say that a Hausdorff compactification $\bar{X}$ of $X$ is perspective if $\varphi$ extends continuously to an action on $\bar{X}$ and $\forall u \in \mathcal{U}, \ \forall K \subseteq X$ compact, $\#\{g \in G: \varphi(g,K)\notin Small(u)\} < \aleph_{0}$, where $\mathcal{U}$ is the only uniform structure compatible with the topology of $\bar{X}$ (given by Theorem 1, $\S 4$, Chapter II of \cite{Bou}). If $\bar{X}$ is metrizable, this condition is equivalent to that $\forall \varepsilon > 0$, $\forall K \subseteq X$ compact, $\#\{g \in G: diam \ \varphi(g,K) > \varepsilon\} < \aleph_{0}$ for any choice of a metric compatible with the topology of $\bar{X}$.

We denote by $EPers_{0}(\varphi)$ the category whose objects are perspective compactifications of $X$ and morphisms are continuous equivariant maps that are the identity when restricted to $X$. If $G = X$ and the action is the left multiplication action then we use the notation $EPers_{0}(G)$ for such category.
\end{defi}

\begin{prop}(Theorem 3.2 of \cite{So}) Let $K \subseteq X$ be a fundamental domain of $\varphi$. The functor $\Pi: EPers(G) \rightarrow EPers(\varphi)$ that sends $G+_{\partial}Y$ to  $X+_{\partial_{\Pi_{K}}}Y, \ L+\psi:G \curvearrowright G+_{\partial}Y$ to $\varphi+\psi: G \curvearrowright X+_{\partial_{\Pi_{K}}}Y$ and $id+\phi: G+_{\partial_{1}}Y_{1}\rightarrow G+_{\partial_{2}}Y_{2}$ to $id+\phi: X+_{(\partial_{1})_{ \Pi_{K}}} Y_{1}\rightarrow X+_{(\partial_{2})_{\Pi_{K}}}Y_{2}$, is a isomorphism of categories.

Furthermore, its inverse is the functor $\Lambda: EPers(\varphi) \rightarrow EPers(G)$ that sends $X+_{f}Y$ to $G+_{f_{\Lambda_{K}}}Y, \ \varphi+\psi:G \curvearrowright X+_{f}Y$ to $id+\psi: G \curvearrowright G+_{f_{\Lambda_{K}}}Y$ and $id+\phi: X+_{f_{1}}Y_{1}\rightarrow X+_{f_{2}}Y_{2}$ to $id+\phi: G+_{(f_{1})_{\Lambda_{K}}}Y_{1}\rightarrow G+_{(f_{2})_{\Lambda_{K}}}Y_{2}$.
\end{prop}

Since $\Pi$ and $\Lambda$ do not depend of the choice of the fundamental domain (Propositions 3.18 and 3.20 of \cite{So}), we denote $\partial_{\Pi_{K}}$ by $\partial_{\Pi}$ and $f_{\Lambda_{K}}$ by $f_{\Lambda}$.

\begin{prop}\label{ezcorrespondence}Let $G$ be a group and $\varphi_{i}: G \curvearrowright X_{i}$ properly discontinuous and cocompact actions, $i = \{1,2\}$, where $X_{1}$ and $X_{2}$ are Hausdorff locally compact metric ANR spaces. If $\Pi_{i}: EPers_{0}(G) \rightarrow EPers_{0}(\varphi_{i})$ and $\Lambda_{i}: EPers_{0}(\varphi_{i}) \rightarrow EPers_{0}(G)$ are the  functors induced by $\varphi_{i}$, then $\Pi_{2} \circ \Lambda_{1}$ and $\Pi_{1}\circ \Lambda_{2}$ preserves $E\mathcal{Z}$-structures. \eod
\end{prop}

\begin{obs}Actually, this proposition says more information. It says that, in this case, this correspondence agrees with  Guilbault and Moran's correspondence \cite{GM}.
\end{obs}

\subsection{Blowups}

\subsubsection{Convergence actions}

\begin{defi}Let $G$ be a group, $X$ a topological space and $\varphi: G\curvearrowright X$ an action by homeomorphisms. We say that $\varphi$ is properly discontinuous if for every compact set $K \subseteq X$, the set $\{g \in G: \varphi(g,K)\cap K \neq \emptyset\}$ is finite. We say that $\varphi$ is a covering action if $\forall x \in X$, there exists an open neighbourhood $U$ of $x$ such that $\forall g \in G- \{1\}$, $\varphi(g,U)\cap U = \emptyset$. We say that $\varphi$ is cocompact if the quotient space $X/\varphi$ is compact.
\end{defi}

\begin{obs}\label{coveringhausdorff}Suppose that $X$ is locally compact and Hausdorff. It is well known that if $\varphi$ is a covering action, then the projection map $X \rightarrow X/\varphi$ is a covering map, if $\varphi$ is properly discontinuous and free (in special if $G$ is torsion-free), then $\varphi$ is a covering action and if $\varphi$ is properly discontinuous then $X/\varphi$ is Hausdorff.
\end{obs}

\begin{defi}Let $\varphi: G \curvearrowright X$ be an action by homeomorphisms. A point $p \in X$ is bounded parabolic if the action $\varphi|_{Stab_{\varphi}p\times X-\{p\}}: Stab_{\varphi}p \curvearrowright X-\{p\}$ is properly discontinuous and cocompact.
\end{defi}

\begin{defi}Let $G$ be a group, $X$ a Hausdorff compact space and $\varphi: G \curvearrowright X$ an action by homeomorphisms. We say that $\varphi$ has the convergence property if for every wandering net $\Phi$ (i.e. a net such that two elements are always distinct) has a subnet $\Phi'$ such that there exists $a,b \in X$ such that $\Phi'|_{X-\{b\}}$ converges uniformly to $a$.
We say that $\Phi'$ is a collapsing net with attracting point $a$ and repelling point $b$.
\end{defi}

\begin{obs}Regardless the definition above allow us to consider these actions, we are not considering $\varphi$ as a convergence action if the set $X$ has cardinality $2$, unless that $G$ is virtually cyclic, $X = Ends(G)$ and $\varphi$ is the action induced by the left multiplication action $G \curvearrowright G$. \end{obs}

\begin{defi}\label{attractorsumcomp}Let $G$ be a group acting properly discontinuously on a locally compact Hausdorff space $X$ and acting on a Hausdorff compact space $Y$ with the convergence property. The attractor-sum compactification of $X$ is the unique compactification of $X$ with $Y$ as its remainder and such that the action of $G$ on it (that extends both actions) has the convergence property. We denote such compactification by $X+_{f_{c}} Y$ (where $c$ means that the action still has the convergence property).
\end{defi}

\begin{obs}The existence and uniqueness of the attractor-sum compactification is due to Gerasimov (Proposition 8.3.1 of \cite{Ge2}).
\end{obs}

\begin{defi}Let $\varphi: G \curvearrowright X$ be a convergence action and $p \in X$. We say that $p$ is a conical point if there is an infinite set $K \subseteq G$ such that $\forall q \neq p$, $Cl_{X}(\{(\varphi(g,p),\varphi(g,q)): g \in K\})\cap \Delta X = \emptyset$.
\end{defi}

\begin{defi}\label{hypgroup}Let $G$ be a group, $X$ a metrizable compact space and $\varphi: G \curvearrowright X$ an action by homeomorphisms. We say that $\varphi$ is hyperbolic if it has the convergence property and every point of $X$ is conical. In this case, we say that $G$ is hyperbolic and $X$ is its hyperbolic boundary. Then we denote it by $X = \partial_{\infty}(G)$.
\end{defi}

\begin{obs}Such space $X$, if it exists, is unique, up to unique equivariant homeomorphism (Lemma 5.1 of \cite{Bo2}).

An equivalent definition is that the minimal convergence action on $X$ induces a cocompact action on the space of distinct triples of $X$ (Theorem 8.1 of \cite{Bo1}). The classical definitions of hyperbolic group and hyperbolic boundary are also equivalent to that, i.e. $X$ is the hyperbolic boundary of a $\delta$-hyperbolic proper geodesic metric space, where $G$ acts properly discontinuously cocompactly and by isometries (Proposition 1.13 of \cite{Bo2} and Theorem 0.1 of \cite{Bo1}. In this present text we need to use just the one stated in the definition above.
\end{obs}

\begin{defi}\label{relhyppair}Let $G$ be a group, $X$ a Hausdorff compact space and $\varphi: G \curvearrowright X$ a minimal action by homeomorphisms. We say that $\varphi$ is relatively hyperbolic if it has the convergence property and its limit set has only conical and bounded parabolic points. If $\mathcal{P}$ is a representative set of conjugation classes of stabilizers of bounded parabolic points of $X$, then we say that $G$ is relatively hyperbolic with respect to $\mathcal{P}$ (or equivalently that $(G,\mathcal{P})$ is a relatively hyperbolic pair). We say that $X$ is is the Bowditch boundary of the pair $(G,\mathcal{P})$ and we denote it by $X = \partial_{B}(G,\mathcal{P})$.
\end{defi}

\begin{obs}If $(G,\mathcal{P})$ is a relatively hyperbolic pair, then its Bowditch boundary is uniquely defined, up to unique equivariant homeomorphism (Corollary 6.1(e) of \cite{GP2}).

An equivalent definition is that the minimal convergence action on $X$ induces a cocompact action on the space of distinct pairs (1C of \cite{Tu} and Main Theorem of \cite{Ge1}).
\end{obs}

\subsubsection{The blowup construction}

Next we show the general idea of the construction of the parabolic blowup. The details are covered in \cite{So3}. However, this general idea and \textbf{Propositions \ref{surjectiveness}} and \textbf{\ref{blowupconv}} are enough information for the subsequent section of this text.

Let $G$ be a group, $Z$ a compact Hausdorff space, $\varphi: G \curvearrowright Z$ an action by homeomorphisms, $P \subseteq Z$ the set of bounded parabolic points of $\varphi$, $P' \subseteq P$ a representative set of orbits, $\mathcal{C} = \{C_{p}\}_{p\in P'}$ a family of compact Hausdorff spaces, $H = \{H_{p}\}_{p\in P'}$, with $H_{p} \subseteq G$ minimal sets such that $1\in H_{p}$ and $Orb_{\varphi|_{H_{p}\times Z}} p = Orb_{\varphi} p$ and $\{Stab_{\varphi} p+_{\partial_{p}}C_{p}\}_{p\in P'}$ a family of spaces with the equivariant perspective property with actions $\eta = \{\eta_{p}\}_{p\in P'}$, such that the action $L_{p}+\eta_{p}: Stab_{\varphi} p \curvearrowright Stab_{\varphi} p +_{\partial_{p}}C_{p}$  is by homeomorphisms (where $L_{p}$ is the left multiplication action).

We define for $p \in P'$ and $q \in Orb_{\varphi} p, \ C_{q} = C_{p}$ and define $h_{p,q}$, for $q \in Orb_{\varphi} p$, the only element of $H_{p}$ such that $\varphi(h_{p,q},p) = q$. Take, for $p \in P'$ and $q \in Orb_{\varphi} p$, the compact space $Stab_{\varphi} p +_{\partial_{q}}C_{q}$, where $\partial_{q} = (\partial_{p})_{\ast}$, is the pullback for  $\tau_{h_{p,q}^{-1}}: Stab_{\varphi}q \rightarrow Stab_{\varphi}p$, the conjugation map $\tau_{h_{p,q}^{-1}}(x) = h_{p,q}^{-1}xh_{p,q}$, and $id_{C_{p}}$.

Since $P$ is the set of bounded parabolic points, we have that $\forall p \in P$,  $\varphi|_{Stab_{\varphi} p \times (Z - \{p\})}$ is proper and cocompact, so we are able to take the space $X_{p} = (Z - \{p\}) +_{(\partial_{p})_{\Pi}} C_{p}$. We have also that the quotient map $\pi_{p}: X_{p} \rightarrow Z$, such that $\pi_{p}|_{Z - \{p\}}$ is the identity map and $\pi_{p}(C_{p}) = p$, is continuous.

Let $X = \lim\limits_{\longleftarrow}(\{X_{p}\}_{p \in P},\{\pi_{p}\}_{p \in P})$. The space $X$, together with the induced action of $G$ on it, is called parabolic blowup of $(Z,\varphi)$ by  $(\mathcal{C},\eta)$.

The next proposition characterizes blowup spaces:

\begin{prop}\label{surjectiveness}(Proposition of \cite{So3}) Let $X$ and $Z$ be Hausdorff compact topological spaces, $\psi: G \curvearrowright X$ and $\varphi: G \curvearrowright Z$ actions by homeomorphisms, $P$ the set of bounded parabolic points of $Z$ and a continuous surjective equivariant map $\pi: X \rightarrow Z$ such that $\forall x \in Z-P$, $\#\pi^{-1}(x) = 1$. For $p \in P$, take $X = W_{p}+_{\delta_{p}}\pi^{-1}(p)$, where $W_{p} = X - \pi^{-1}(p)$ and $\delta_{p}$ is the only admissible map that gives the original topology. Let also decompose $\psi|_{Stab_{\varphi} p \times X}$ as $\psi_{p}^{1}+\psi_{p}^{2}$, with respect to the decomposition of the space. Those are equivalents:

 \begin{enumerate}

    \item The map $\pi$ is topologically quasiconvex and $\forall p \in P$, $X \in EPers(\psi_{p}^{1})$.

    \item There exists an equivariant homeomorphism $T$ to a parabolic blowup $Y$ that commutes the diagram:

$$ \xymatrix{ X  \ar[rd]^{\pi} \ar[d]_{T} &  \\
             Y \ar[r]_{\pi'} & Z} $$

Where $\pi'$ is the projection map.

\end{enumerate}

\end{prop}

In this case the map $\pi$ is called a blowup map.

If the action $\psi$ have the convergence property, then the conditions in (1) are not necessary, i.e. every surjective equivariant continuous map $\pi$ such that $\forall x \in Z-P$, $\#\pi^{-1}(x) = 1$ is a blowup map (Theorem 3.1 of \cite{So3}). In the special case that $\psi$ have the convergence property and $\varphi$ is relatively hyperbolic, we have that every surjective equivariant continuous map is a blowup map (Proposition 3.2 of \cite{So3}), since every limit point that is not bounded parabolic is conical and preimage of a conical point is a single point (Proposition 7.5.2 of \cite{Ge2}).

If the action $\varphi$ is relatively hyperbolic, then we have more properties:

\begin{prop}\label{blowupconv}Let $G$ be a finitely generated group, $X$ and $Z$ Hausdorff compact spaces $\psi: G \curvearrowright X$ and $\varphi: G \curvearrowright Z$ actions by homeomorphisms, $P \subseteq Z$ be the set of bounded parabolic points and $\pi: X \rightarrow Z$ be a blowup map. If $\varphi$ is minimal and relatively hyperbolic, then we have:

\begin{enumerate}
    \item (Theorem 3.9 of \cite{So3}) If $\forall p\in P$, the action of $Stab_{\varphi} \ p$ on $\pi^{-1}(p)$ have the convergence property, then $\psi$ has the convergence property.
    \item (Corollary 3.15 of \cite{So3}) If $\forall p\in P$, the action of $Stab_{\varphi} \ p$ on $\pi^{-1}(p)$ is minimal and hyperbolic, then $\psi$ is hyperbolic (and then $X \cong \partial_{\infty}G$).
    \item (Corollary 3.17 of \cite{So3}) If $\forall p\in P$, the action of $Stab_{\varphi} \ p$ on $\pi^{-1}(p)$ is minimal and relatively hyperbolic with respect to the parabolic subgroups $\mathcal{P}_{p}$, then $\psi$ is relatively hyperbolic with respect to the set of parabolic subgroups $\mathcal{P} = \bigcup\{\mathcal{P}_{p}: p \in P\}$ (and then $X \cong \partial_{B}(G,\mathcal{P})$).
\end{enumerate}

\end{prop}

\begin{obs}Theorem 3.9 of \cite{So3} is stated for geometric compactification instead of minimal relatively hyperbolic action, which is more general.

Corollary 3.17 of \cite{So3} has as immediate consequence a Drutur and Sapir Theorem (Corollary 1.14 of \cite{DS}) which says that if $G$ is relatively hyperbolic with respect to $\mathcal{P}$ and every group H in $\mathcal{P}$ is relatively hyperbolic with respect to $\mathcal{P}_{H}$, then $G$ is relatively hyperbolic with respect to $\bigcup_{H \in \mathcal{P}} \mathcal{P}_{H}$.

\end{obs}

\section{Blowing up spheres}
\label{blowingupspheres}

\begin{prop}(Tshishiku-Walsh \cite{TW2} for $n = 2$, de S. \cite{So4} for $n \neq 2,4$) Let $P$ be a countable dense subset of the sphere $S^{n}$, with $n \neq 4$. Let, for every $p \in P$, $\pi_{p}: D^{n} \rightarrow S^{n}$ be the map that collapses the boundary of the $n$-ball to the point $p$. Then the limit space $\lim\limits_{\longleftarrow} \{D^{n},\pi_{p}\}_{p \in P}$ is homeomorphic to the $n-1$-dimensional Sierpi\'nski carpet.
\end{prop}

\begin{obs}It is not clear if this proposition holds for $n = 4$.
\end{obs}

\begin{prop}\label{blowupsierpisnki}Let $G$ be a group that has an action by homeomorphisms on $S^{n}$. Let $P \subseteq S^{n}$ be the set of bounded parabolic points, $P' \subseteq P$ be a set of representatives and $\pi: X \rightarrow S^{n}$ a blowup map. Let $Q' \subseteq P'$, $Q = Orb_{\varphi} Q' \subseteq P$ and suppose that $Q$ is countable and dense on $S^{n}$. If $n = 4$, then we also suppose that the last proposition holds for dimension $4$. If $\forall p \in Q'$, $X_{p}$ is the closed ball $D^{n}$ with the boundary equal to the preimage of $p$ and $\forall p \in P' - Q'$, $X_{p}$ is homeomorphic to $S^{n}$ and the map to $S^{n}$ is a homeomorphism, then $X$ is homeomorphic to the Sierpi\'nski carpet of dimension $n-1$.
\end{prop}

\begin{proof}It is immediate from the last proposition.
\end{proof}

Tshishiku and Walsh showed that if $(G,\mathcal{P})$ is a relatively hyperbolic pair with its Bowditch boundary homeomorphic to the sphere $S^{2}$ and $X$ then the Dahmani boundary $\partial_{D}(G,\mathcal{P})$ is the $1$-dimensional Sierpi\'nski carpet (Theorem 1.1 of \cite{TW}). Our methods used next are similar to theirs.

What we do next is to establish the necessary condition to prove the \textbf{Theorem \ref{sierpinskihyp}}.

\begin{lema}\label{finiteindexhyp}Let $G$ be a group, $H$ a finite index subgroup of $G$, $X$ a proper $\delta$-hyperbolic space, $\varphi: G \curvearrowright X$ a properly discontinuous cocompact action and $\varphi'$ the restriction of $\varphi$ to $H$. If $\varphi'$ acts by isometries, then the induced functor $\Pi$ sends the hyperbolic compactification of $G$ to the hyperbolic compactification of $X$.
\end{lema}

\begin{obs}$G$ is hyperbolic since $H$ is hyperbolic and it has a finite index.
\end{obs}

\begin{proof}The functor $\Pi$ sends the hyperbolic compactification $G+_{\partial_{\infty}}\partial_{\infty}(G)$ to a space $W = X+_{f}\partial_{\infty}(G)$, and the action of $G$ on both has the convergence property (by the construction of $\Pi$ and Proposition 8.3.1 of \cite{Ge2}). It implies that the action of $H$ on $W$ has the convergence property.  But $\partial_{\infty}(G) = \partial_{\infty}(H)$ (since $H$ has finite index over $G$), which implies that $W = X+_{f}\partial_{\infty}(G)$ is the only compactification of $X$ such that the action of $H$ on it extends $\varphi'$ and it has the convergence property and the restriction to the action of $H$ on the boundary is hyperbolic (see \textbf{Definitions \ref{attractorsumcomp}} and \textbf{\ref{hypgroup}}). However, the hyperbolic compactification of $X$ has also this property, since $H$ acts on $X$ properly discontinuously, cocompactly and by isometries. So $W$ is the hyperbolic compactification of $X$.
\end{proof}

\begin{defi}Let $M$ be a compact manifold without boundary. We say that $M$ is aspherical if $\pi_{n}(M,p)$ is trivial for every $n > 1$. We say that $M$ is topologically rigid if for every compact manifold without boundary $N$, every homotopic equivalence $M \rightarrow N$ is homotopic to a homeomorphism.
\end{defi}

\begin{conj}Let $G$ be a group. We say that the Borel Conjecture holds for $G$ if every compact aspherical manifold without boundary such that $\pi_{1}(M,p) = G$ is topologically rigid.
\end{conj}

\begin{obs}See more about the conjecture in \cite{BL}.
\end{obs}

\begin{prop}(Ferry, \cite{Fe}) If $X$ is an AR space and $Z\subseteq X$ is a $\mathcal{Z}$-set such that $Z$ is homeomorphic to $S^{n-1}$ and $X-Z$ is homeomorphic to $\R^{n}$, with $n \geqslant 5$, then $X$ is homeomorphic to the closed  $n$-ball.
\end{prop}

\begin{obs}It is stated in page 479 of \cite{BM}.
\end{obs}

\begin{prop}Let $G$ be a group $\varphi: G \curvearrowright S^{n}$ an action by homeomorphisms, with $n \geqslant 5$, $P$ the set of bounded parabolic points, $P' \subseteq P$ a set of representatives, $Q' \subseteq P'$, $Q = Orb_{\varphi} Q' \subseteq P$,  and $\eta = \{\eta_{p}\}_{p \in P'}$ a family of actions with $L_{p}+\eta_{p}: Stab_{\varphi}p \curvearrowright Stab_{\varphi}p+_{\delta_{p}}C_{p}$ such that $Stab_{\varphi}p+_{\delta_{p}}C_{p}$ has the perspectivity property (where $L_{p}$ is the left multiplication action), $C_{p}$ is a single point if $p \notin Q'$ and $C_{p}$ is homeomorphic to $S^{n-1}$ otherwise. Suppose that $Q$ is countable and dense on $S^{n}$. If, for every $p \in Q'$, there exists a space $E_{p}+_{f_{p}}C_{p}$ that forms an $E\mathcal{Z}$-structure for $Stab_{\varphi}p$ that it is in correspondence with $Stab_{\varphi}p+_{\delta_{p}}C_{p}$ by the functor induced by the action of $Stab_{\varphi}p$, then the blowup space $X$ with respect to $\eta$ is homeomorphic to a $n-1$-dimensional Sierpi\'nski carpet.
\end{prop}

\begin{proof}Let $p \in P' - Q'$.  Let's name the actions $\varphi_{p} = \varphi|_{ Stab_{\varphi} p \times (S^{n} - \{p\})}$ and $\psi_{p}+\eta_{p}: Stab_{\varphi} p \curvearrowright E_{p}+_{f_{p}}C_{p}$ (the action given by the $E\mathcal{Z}$-structure). If $\Pi_{p}: EPers(Stab_{\varphi}p) \rightarrow EPers(\varphi_{p})$ and $\Lambda'_{p}: EPers(\psi_{p}) \rightarrow EPers(Stab_{\varphi}p)$ are the functors induced by $\varphi_{p}$ and $\psi_{p}$, respectively, then $\Pi_{p}\circ\Lambda'_{p}$ sends $E\mathcal{Z}$-structures to $E\mathcal{Z}$-structures (\textbf{Proposition \ref{ezcorrespondence}}). So the space $X_{p} = \Pi_{p}(Stab_{\varphi}p+_{\delta_{p}}C_{p}) = (S^{n} - \{p\}))+_{f'_{p}}C_{p}$, where $f'_{p}$ is  some admissible map, form an $E\mathcal{Z}$-structure of $Stab_{\varphi}p$ with $\mathcal{Z}$-boundary the subspace $C_{p}$. But $S^{n} - \{p\}$ is homeomorphic to $\R^{n}$ and $C_{p}$ is homeomorphic to $S^{n-1}$, which implies, by Ferry's Theorem, that $X_{p}$ is homeomorphic to the closed $n$-ball. Thus, it follows from \textbf{Proposition \ref{blowupsierpisnki}} that $X$ is a $n-1$-dimensional Sierpi\'nski carpet.
\end{proof}

Then we have:

\begin{teo}\label{sierpinskihyp}Let $(G,\mathcal{P})$ be a relatively hyperbolic pair with Bowditch boundary homeomorphic to $S^{n}$ and $\mathcal{Q}$ a proper subset of $\mathcal{P}$ such that every group in $\mathcal{P} - \mathcal{Q}$ is hyperbolic. Suppose that one of the following holds:

\begin{enumerate}
    \item $n = 1$ or $2$.
    \item $n = 3$ and every group in $\mathcal{P}-\mathcal{Q}$ is virtually torsion-free.
    \item $n = 4$, the topological characterization of the $3$-dimensional Sierpi\'nski Carpet as an inverse limit of $4$-balls holds, every group $Q \in \mathcal{P} - \mathcal{Q}$ is torsion-free, it acts properly discontinuously, cocompactly and by isometries on a simply connected manifold $M_{Q}$ with negatively bounded curvature and the Borel Conjecture holds for $Q$.
    \item $n \geqslant 5$ and every group in $\mathcal{P} - \mathcal{Q}$ has hyperbolic boundary homeomorphic to $S^{n-1}$.
\end{enumerate}

Then $\partial_{B}(G,\mathcal{Q})$ is a $n-1$-dimensional Sierpi\'nski carpet.
\end{teo}

\begin{obs}$G$ is relatively hyperbolic with respect to $\mathcal{Q}$ by a theorem of Drutu and Sapir  (Corollary 1.14 of \cite{DS}).

If $n = 1$, then every group $P\in\mathcal{P}$, stabilizer of a bounded parabolic point $p$, acts properly discontinuously and cocompactly on $\partial_{B}(G,\mathcal{P}) - \{p\}$, which is homeomorphic to $\R$. Then $P$ is two-ended, which implies that it is hyperbolic.

If the Dahmani boundary is equivalent with the parabolic blowup with a Bowditch boundary as a base, then this result for $n = 2$ is a special case of Theorem 1.1 of \cite{TW}.
\end{obs}

\begin{proof}Let $X$ be the blowup space of $\partial_{B}(G,\mathcal{P})$ with fibers the hyperbolic boundary of the groups in $\mathcal{P} - \mathcal{Q}$ and a single point for each group in $\mathcal{Q}$. We have that $X = \partial_{B}(G,\mathcal{Q})$, by \textbf{Proposition \ref{blowupconv}}.

Suppose $n = 1$

Let $Q \in \mathcal{P} - \mathcal{Q}$ be the stabilizer of a bounded parabolic point $q$. Since $X_{q}$ is a compactification of  $\partial_{B}(G,\mathcal{P}) - \{q\}$ with $\partial_{\infty}(Q)$ as boundary, $\partial_{B}(G,\mathcal{P}) - \{q\}$ is homeomorphic to $\R$, $\#\partial_{\infty}(Q) = 2$ and the only compactification of $\R$ with two points is homeomorphic to the closed $1$-ball, it follows that $X_{q}$ is homeomorphic to the $1$-ball. Thus, by \textbf{Proposition \ref{blowupsierpisnki}} the space $\partial_{B}(G,\mathcal{Q})$ is a $0$-dimensional Sierpi\'nski carpet (i.e. a Cantor set).

Suppose $n = 2$.

Let $Q \in \mathcal{P} - \mathcal{Q}$ be the stabilizer of a bounded parabolic point $q$. The action of $Q$ on $\partial_{B}(G,\mathcal{P})-\{q\}$ is properly discontinuous, which implies that it has a finite kernel $K$. The quotient $R = Q/K$ has an induced action on $\partial_{B}(G,\mathcal{P})-\{q\}$ that is faithful. By Theorem 1.3 of \cite{HL}, there are finite index subgroups $Q' < Q$ and $R' < R$ such that the restriction of the quotient map is an isomorphism between $Q'$ and $R'$. Let $\varphi_{q}: Q' \curvearrowright \partial_{B}(G,\mathcal{P})-\{q\}$ be the restriction of the action of $G$ on $\partial_{B}(G,\mathcal{P})$. By the construction of $Q'$, we have that $\varphi_{q}$ is a faithful action (and it is cocompact, since $Q'$ has finite index over $Q$).

By Theorem 0.3 of \cite{Da3} all elements of $\mathcal{P}$ are virtual surface groups. By taking a finite index subgroup, we can suppose that $Q'$ is a surface group (which is torsion-free). So $(\partial_{B}(G,\mathcal{P})-\{q\})/\varphi_{q}$ is a compact $2$-manifold, with fundamental group isomorphic to $Q'$. Since $Q'$ is hyperbolic, then it is the fundamental group of a hyperbolic surface. By the classification of compact surfaces and existence of Riemannian metrics of constant curvature on compact surfaces, there exists a Riemannian metric on $\partial_{B}(G,\mathcal{P})-\{q\}$ with constant negative curvature such that the action $\varphi_{q}$ is by isometries. So $\partial_{\infty}(Q') = \partial_{\infty}(Q)$ is homeomorphic to $S^{1}$ and the hyperbolic compactification of $\partial_{B}(G,\mathcal{P})-\{q\}$ is a closed $2$-ball. But such compactification is $X_{q}$ (of the construction of the blowup space) since the functor induced by the action of $Q$ preserves hyperbolic compactifications, by the \textbf{Lemma \ref{finiteindexhyp}}.

Suppose $n = 3$.

Let $Q \in \mathcal{P} - \mathcal{Q}$ be the stabilizer of a bounded parabolic point $q$, $Q'$ a finite index torsion-free subgroup of $Q$ and $\varphi_{q}: Q' \curvearrowright \partial_{B}(G,\mathcal{P})-\{q\}$ the restriction of the action of $G$ on $\partial_{B}(G,\mathcal{P})$. Since $Q'$ is torsion-free, $\varphi_{q}$ is a covering action (and it is already cocompact).

Since $\partial_{B}(G,\mathcal{P})-\{q\}$ is connected and locally connected and $\varphi_{q}$ is a covering action, there exists, by Proposition 6.3 of \cite{GM}, a metric $d$ on the space $\partial_{B}(G,\mathcal{P})-\{q\}$ such that the metric space is proper and geodesic and $\varphi_{q}$ is an action by isometries. Since $Q'$ is hyperbolic, then $\partial_{B}(G,\mathcal{P})-\{q\}$ is quasi-isometric to $Q'$ and then it is also metric hyperbolic.

The quotient space $(\partial_{B}(G,\mathcal{P})-\{q\})/\varphi_{q}$ is a compact $3$-manifold, with fundamental group isomorphic to $Q'$ and irreducible (i.e. every locally flat $2$-sphere in it bounds a closed $3$-ball), since its universal cover, $\partial_{B}(G,\mathcal{P})-\{q\} \cong \R^{3}$, is irreducible (Proposition 1.6 of \cite{Ha} and the Generalized Schoenflies Theorem \cite{Br}). So, by Theorem 4.1 of \cite{BM}, $\partial_{\infty}(Q') = \partial_{\infty}(Q)$ is homeomorphic to $S^{2}$ and the hyperbolic compactification of $\partial_{B}(G,\mathcal{P})-\{q\}$ is a closed $3$-ball. But such compactification is $X_{q}$ since the functor induced by the action of $Q$ preserves hyperbolic compactifications (again, by the \textbf{Lemma \ref{finiteindexhyp}}).

Suppose that $n = 4$.

Let $Q \in \mathcal{P} - \mathcal{Q}$ be the stabilizer of a bounded parabolic point $q$ and $\varphi_{q}: Q \curvearrowright \partial_{B}(G,\mathcal{P})-\{q\}$ the restriction of the action of $G$ on $\partial_{B}(G,\mathcal{P})$. Since $Q$ is torsion-free, $\varphi_{q}$ is a covering action (and it is already cocompact).

Since $\varphi_{q}$ is a covering action, properly discontinuous and cocompact, the quotient space $(\partial_{B}(G,\mathcal{P})-\{q\})/\varphi_{q}$ is a compact $4$-manifold, with fundamental group isomorphic to $Q$. Its universal cover $\partial_{B}(G,\mathcal{P})-\{q\} \cong \R^{4}$ is contractible, which implies that $(\partial_{B}(G,\mathcal{P})-\{q\})/\varphi_{q}$ is a $K(Q,1)$ space. We also have that $M_{Q}/Q$ is a compact manifold and a $K(Q,1)$ space. Since the Borel Conjecture holds for $Q$, the spaces $(\partial_{B}(G,\mathcal{P})-\{q\})/Q$ and $M_{Q}/Q$ are homeomorphic. So there exists a Riemannian metric on $\partial_{B}(G,\mathcal{P})-\{q\}$ which have negatively bounded curvature and the action  $\varphi_{q}$ is by isometries (we can do it in a way that $\partial_{B}(G,\mathcal{P})-\{q\}$ and $M_{Q}$ are isometric). Then the hyperbolic compactification of $\partial_{B}(G,\mathcal{P})-\{q\}$ is a closed $4$-ball. But
such compactification is $X_{q}$ since the functor induced by the action of $Q$ preserves hyperbolic compactifications.

In any of those cases it follows from \textbf{Proposition \ref{blowupsierpisnki}} that the space $\partial_{B}(G,\mathcal{Q})$ is a $n-1$-dimensional Sierpi\'nski carpet.

Suppose $n \geqslant 5$.

Let $Q \in \mathcal{P} - \mathcal{Q}$. Since it is hyperbolic, it has an $E\mathcal{Z}$-structure given by its Rips complex $P_{d}$, for a big enough $d > 0$, compactified with its hyperbolic boundary (Theorem 1.2 of \cite{BM}). This compactification is in correspondence with $Q+_{\partial_{\infty}}\partial_{\infty}(Q)$ since the functor induced by the action of $Q$ preserves hyperbolic compactifications. Then, by the last proposition, the blowup $X$ is homeomorphic to the $n-1$-dimensional Sierpi\'nski carpet. But $X = \partial_{B}(G,\mathcal{Q})$.
\end{proof}

\begin{obs}Theorem 1.3 of \cite{HL} says only that there are finite index subgroups $Q' < Q$ and $R' < R$ that are isomorphic. But this isomorphism comes from the restriction of the quotient map, as seen on its proof. In the statement of Theorem 1.3 of \cite{HL} it is missing the hypothesis of the map being surjective, which is the case that we needed above.
\end{obs}

\begin{cor}\label{circleimpliesvirtuallyfree}Let $(G,\mathcal{P})$ be a relatively hyperbolic pair with Bowditch boundary homeomorphic to $S^{1}$ and such that $\mathcal{P}$ is not empty. Then $G$ is virtually free.
\end{cor}

\begin{obs}This corollary is well known.
\end{obs}

\begin{proof}Since every group in $\mathcal{P}$ is hyperbolic, then $G$ is hyperbolic (\textbf{Proposition \ref{blowupconv}}). It follows from the last theorem that $\partial_{\infty}(G)$ is a Cantor set. Thus, by 7.A.1 of \cite{Gr}, $G$ is virtually free.
\end{proof}

Examples as the blowups given in \textbf{Theorem \ref{sierpinskihyp}} appear in every dimension in the base:

\begin{ex}Let $M$ be a compact real hyperbolic $n+1$-manifold with non-empty totally geodesic boundary and $n \neq 4$. We have that its fundamental group $G = \pi_{1}(M,p)$ is hyperbolic and $\partial_{\infty}(G)$ is a $n-1$-dimensional Sierpi\'nski carpet (Theorem 1.1 of \cite{La}). Let $\{C_{i}\}_{i \in \N}$ be the set of the $n-1$-spheres in $\partial_{\infty}(\pi_{1}(M,p))$ that appears from the boundaries of the balls removed from $S^{n}$ on the standard definition of the $n-1$-dimensional Sierpi\'nski carpet and $\mathcal{P}$ a set of conjugation classes of the stabilizers of the sets $\{C_{i}\}_{i \in \N}$. Let $P \in \mathcal{P}$ be the stabilizer of the sphere $C$. We have that $C$ is $P$-invariant,  $\mathcal{P}$ is finite (Proposition 3.4 of \cite{Da3}), $C$ is the limit set of $P$ (Lemma 3.5 of \cite{Da3}) and $P$ is a quasi-convex subgroup of $G$ (Proposition 3.7 of \cite{Da3})  (all proofs work in every dimension). Since the $n-1$-spheres in $\{C_{i}\}_{i \in \N}$ are pairwise disjoint, we have that the family $\mathcal{P}$ is almost malnormal (i.e. $\forall P,P' \in \mathcal{P}$, $\forall g \in G$, $P = gP'g^{-1}$ or $P \cap gP'g^{-1} = \emptyset$). If $g \in G$, the limit set of $ gPg^{-1}$ is the limit set of $gP$ (using the perspectivity property) and the limit set of $gP$ is $gC$, which implies that the limit set of $ gPg^{-1}$ is $gC$  and then $gPg^{-1} = P$ if and only if $g \in P$ (in another words, $P$ is its own normalizer). Then $(G,\mathcal{P})$ is a relatively hyperbolic pair (Theorem 7.11 of \cite{Bo4}) with Bowditch boundary given by the quotient of $\partial_{\infty}(G)$ when we collapse all spheres in $\{C_{i}\}_{i \in \N}$ (Main Theorem of \cite{Tr}), which is a $n$-sphere (Decomposition Theorem - Theorem 2 of \cite{Me}). So this quotient map $\partial_{\infty}(G) \rightarrow \partial_{B}(G,\mathcal{P})$ is an example of a blowup map.

In dimension $n = 2$, the existence of examples of hyperbolic  groups with hyperbolic boundary homeomorphic to the $1$-dimensional Sierpi\'nski carpet that are not virtually the fundamental group of a compact hyperbolic $3$-manifold with non-empty totally geodesic boundary is the negation of the Kapovich-Kleiner Conjecture \cite{KK} and it is still open.
\end{ex}

If the Kapovich-Kleiner Conjecture is true, then we have the following proposition which is similar to the Relative Cannon Conjecture:

\begin{cor}\label{conjecturedependent}Let $(G,\mathcal{P})$ be a relatively hyperbolic pair with Bowditch boundary homeomorphic to $S^{2}$ and such that $\mathcal{P}$ is not empty and every group in $\mathcal{P}$ is hyperbolic. If the Kapovich-Kleiner Conjecture holds, then $G$ is virtually the fundamental group of a compact hyperbolic $3$-manifold with non-empty totally geodesic boundary. \eod
\end{cor}

\end{document}